\documentclass{bcp}
\usepackage{amsmath}
\usepackage{amssymb}
\usepackage{mathrsfs}
\usepackage[mathcal]{eucal}
\newtheorem{Thm}{Theorem}[section]
\newtheorem{Prop}[Thm]{Proposition}
\newtheorem{Lem}[Thm]{Lemma}
\newtheorem{Cor}[Thm]{Corollary}

\theoremstyle{definition}
\newtheorem{Def}[Thm]{Definition}
\newtheorem{Ex}[Thm]{Example}
\newtheorem{Rem}[Thm]{Remark}

\def\le{\leqslant}
\def\ge{\geqslant}
\def\C{{\mathbb C}}
\def\R{{\mathbb R}}
\def\Q{{\mathbb Q}}
\def\Z{{\mathbb Z}}
\let\parasymbol=\S
\def\secref#1{\parasymbol\ref{#1}}
\def\S{\varSigma}
\def\ssm{\smallsetminus}
\def\d{\mathrm d}
\edef\polishl{\l}
\def\l{\lambda}
\edef\polishL{\L}
\def\Poly{\operatorname{Poly}}
\def\e{\varepsilon}
\def\bigland{\textstyle\bigwedge\nolimits}
\def\G{\varGamma}
\def\f{\varphi}

\begin{document}

\mathclass{Primary 34C08, 32S65, 34M03, 34M56; Secondary 34C07, 34M35, 34M99.}

\abbrevauthors{Gal Binyamini e.a.}
\abbrevtitle{Quasialgebraic functions}

\title{Quasialgebraic functions}

\author{G. Binyamini, D. Novikov, S. Yakovenko}

\address{Department of Mathematics,\\
Weizmann Institute of Science,\\
Rehovot 76100, Israel\\
email: \textup{\texttt{\{galbin,dnovikov,yakov\}@weizmann.ac.il}}
}

\maketitlebcp

\begin{abstract}
We introduce and discuss a new class of (multivalued analytic) transcendental functions which still share with algebraic functions the property that the number of their isolated zeros can be explicitly counted. On the other hand, this class is sufficiently rich to include all periods (integral of rational forms over algebraic cycles).
\end{abstract}

\section{Motivations}

The question about the number and location of (isolated) zeros for different classes of analytic functions appears in connection with numerous problems, both in geometry (enumeration problems) and analysis. For instance, the famous  Hilbert's Sixteenth problem (one of the two in the Hilbert list which still remains unsolved) on limit cycles of planar polynomial vector fields can be reformulated as a question on isolated zeros of the Poincar\'e displacement function (see below for more details).

The obvious class of functions for which the counting problem admits a universally known, simple and precise answer, is that of polynomials (in one complex variable). The Fundamental Theorem of Algebra states that the number of isolated complex roots of a polynomial of degree $d>0$ is always between $1$ and $d$, and is exactly equal to $d$ if the roots are counted with multiplicity.

An instant generalization of this example is given by the class of algebraic functions of one variable: the required bounds are provided by the B\'ezout theorem. Because of the ramification and multivaluedness, the counting problem requires certain precautions to state. In this particular case it is sufficient to triangulate the complement $\C P^1\ssm\S$ to the finite ramification locus $\S$, count the number of isolated zeros of each branch of the function in each simplex (triangle) of the triangulation, and add up the separate results. As with the polynomials, if all roots are counted properly, the exact answer for the total number of isolated zeros is given by the intersection theory in purely algebraic terms.

However, essentially broader classes  of such type are not so easily produced. For instance, as was argued in \cite{montr}, no nontrivial class admitting finite count of zeros, can exist as a  subclass of single-valued functions on $\C P^1\ssm\S$ with any finite or infinite locus $\S$. Indeed, if all excluded points are poles of finite order (and their number is finite), then such a function $f$ is necessarily rational. If, on the contrary, one of the points  $a\in\S$ is an essential singularity for $f$ (the absolute value $|f(t)|$ grows to infinity faster than polynomially as $t\to a$), then by the Picard theorem, $f$ assumes almost all values infinitely many times in any neighborhood of $a$.

Thus any nontrivial class of functions allowing for an explicit global counting of their zeros on $\C P^1$ should necessarily include multivalued functions. This necessarily brings about the question on the automorphy (monodromy) of these functions, that is, how different branches obtained by analytic continuation avoiding the locus $\S$ are related to each other.

The most natural assumption, which covers a huge number of possible applications, is the assumption that the linear span generated by all branches of these functions, is a well-defined finite-dimensional linear space near any point outside $\S$. In particular, this includes all algebraic functions whose branches are simply permuted by the monodromy transformations. By the classical Riemann-type arguments \cite{thebook}, these spaces can be identified with the spaces of solutions of linear ordinary differential equations (or systems of such equations) with single-valued  coefficients holomorphic on $\C P^1\ssm\S$.

If a singular point $a\in\S$ is an isolated essential singularity for the coefficients of the equation, then one can again argue that solutions of the equation will grow faster then polynomially (when approaching $a$ along suitable narrow sectors, to avoid the problems with multivalued continuation). Then by the general principles of the Nevanlinna theory, these solutions will have infinite numbers of isolated zeros accumulating to $a$. Thus without loss of generality, we may assume that all coefficients of the linear equations have only finite order poles, i.e., are rational functions of the independent variable $t$ with the polar locus $\S\subset\C P^1$. However, even this class should be further reduced to expect finite count for isolated roots, as explained in the next section.

\section{Counting roots of solutions of linear systems}

\subsection{Crash course on linear systems}
We recall a few basic facts on systems of linear ordinary differential equations; all these facts can be found in \cite[Chapter III]{thebook}.

A system of linear differential equations with rational coefficients on $\C P^1$ can be defined using a $(n\times n)$-matrix $\Omega=\{\Omega_{ij}\}$ of rational 1-forms on $\C P^1$, by the Pfaffian equations
\begin{equation}\label{pf1}
    \d X_i-\sum_{j=1}^n \Omega_{ij}X_j,\quad i=1,\dots, n,\qquad\text{or in the matrix form}\qquad \d X-\Omega X=0.
\end{equation}
In an affine chart $t\in\C\subset\C P^1$ the matrix 1-form $\Omega$ can be written as $A(t)\,\d t$ with a rational matrix function $A(t)=\{A_{ij}(t)\}$ so that the Pfaffian system \eqref{pf1} takes the form of a system of linear first order differential equations $\displaystyle\frac{\d X}{\d t}=A(t)X$. The (``fundamental'') solution of this system is a nondegenerate multivalued analytic matrix function $X=X(t)$, $\det X(t)\ne0$ for $t\notin\S$, ramified over the polar locus $\S$ of the coefficients; the point at infinity may or may not be singular.

The fundamental solution of the linear system \eqref{pf1} is defined uniquely modulo a constant right matrix factor $X(t)\mapsto X(t)C$, $C\in\operatorname{GL}(n,\C)$, hence the linear space spanned by all entries $X_{ij}(t)$ of any solution, is independent of the choice of this solution. Therefore for any simply connected subset $T\subseteq\C P^1\ssm\S$, the linear span is a well-defined $n^2$-dimensional subspace $\mathscr L_T$ in the space $\mathscr O(T)$ of holomorphic functions on $T$.

For a closed path (loop) $\gamma:[0,1]\to\C P^1\ssm\S$, $\gamma(0)=\gamma(1)$, the result of analytic continuation of a nondegenerate matrix solution $X(t)$ along this path is another solution $X'(t)=X(t)M$ where $M$ is a constant matrix factor.  This factor is called the \emph{monodromy matrix} associated with the loop $\gamma$. It is defined by the free homotopy classes of paths up to the conjugacy $M\mapsto C^{-1}MC$ independently of the solution $X(t)$. In particular, if $a_i\in\S$ and $\gamma$ is a simple loop contractible in $(\C P^1\ssm\S)\cup\{a_i\}$, then the corresponding monodromy is called (modulo the conjugacy) the \emph{monodromy of the small loop around} $a_i\in\S$.

\begin{Ex}[principal]\label{ex:euler}
Consider the Euler system with only two singularities, $t=0$ and $t=\infty$, both of which are simple poles of the coefficients. In the affine chart such a system takes the form $\displaystyle\frac{\d X}{\d t}=\frac At\cdot X$, where the constant \emph{residue matrix} $A$ can be assumed to have a Jordan normal form (i.e., diagonal). Solution of this system is a matrix function $t^A=\exp(A\ln t)$. In the case $A$ is diagonal, the space spanned by the components of the matrix $X$, is the space of linear combinations of the power functions $\sum_j c_j\,t^{\l_j}$, $j=1,\dots,n$, where $\l_1,\dots,\l_n$ are the eigenvalues of $A$ (in the general case one will have also the functions $t^{\l_j}\ln^k t$ for multiple eigenvalues).
\end{Ex}

For a general linear system with rational coefficients there is a well known classification of singular points in terms of the growth rate of the corresponding solutions. This classification is parallel to the classification of singularities of single-valued functions into poles and essential singularities: if solutions of a system grow at most polynomially near a point $a\in\S$, then this point is called a ``regular'' (moderate) singularity, otherwise the singularity is ``irregular'' (wild). The only caveat requires that the growth estimate for multivalued (ramified) functions be valid in sectors with the vertex at $a$ (and not in an arbitrary simply connected domain which may be spiraling around $a$). An elementary calculation shows that if $a$ is a \emph{simple} (first order) pole of the matrix form $\Omega$, then it is necessarily a moderate singularity, like the points $0,\infty$ for the Euler system. First order poles are called \emph{Fuchsian singularities}, and the Pfaffian system \eqref{pf1} is called \emph{Fuchsian}, if it has only Fuchsian singularities on $\C P^1$. In any affine chart a Fuchsian system takes the form
\begin{equation}\label{fs}
    \frac{\d X}{\d t}=\biggl(\sum_{j=1}^m \frac{A_j}{t-a_j}\biggr)X, \qquad a_j\in\S,\ A_j\in\operatorname{GL}(n,\C),
\end{equation}
where $A_1,\dots,A_m$ are the \emph{residue matrices} at the singular points $a_1,\dots,a_m$. The point $t=\infty$ is singular if and only if $A_1+\cdots+A_m\ne 0$ (in the singular case the residue at infinity is the negative of this sum).

In general the inverse statement is not true and a moderate singularity may well be non-Fuchsian. However, by a suitable \emph{local meromorphic gauge transformation} a moderate (regular) singularity can be brought into the Fuchsian form. This means that if $a\in\S$ is a moderate singular point, then there exists a neighborhood  $U\owns a$ in $\C P^1$ and a matrix function $H(t)$ holomorphic and invertible in $U\ssm a$ having (together with $H^{-1}(t)$) at most a pole at $a$, such that the matrix 1-form $\d H\cdot H^{-1}+H\Omega H^{-1}$ has a Fuchsian singularity (first order pole) at $a$. This form is the logarithmic derivative $\d Y\cdot Y^{-1}$ of the matrix function $Y(t)=H(t) X(t)$ (``linear change of the dependent variables''), provided that $\Omega=\d X\cdot X^{-1}$ is the logarithmic derivative of $X$.

The problem of \emph{global} classification of regular systems is closely related to the Hilbert 21st Problem. It was discovered only relatively recently (by A. Bolibruch in 1990) that in general not every regular system can be brought to the Fuchsian form by a \emph{rational} gauge transformation $H$ with singularities only at the polar set: there is a subtle obstruction for systems with a reducible monodromy group. However, this is always possible if one allows creation of a single extra singular point outside $\S$. Thus without loss of generality we may always assume that a regular system can be brought into the form \eqref{fs} by a rational gauge transformation.\label{p:gauge}

\begin{Rem}
Together with the analytic description given above, the regular systems can be described geometrically. Given a regular system, with any simply connected open domain $T_\alpha\subset \C P^1\ssm\S$ we can associate the $n$-dimensional linear subspace $\mathscr L_\alpha$ in the space of analytic functions $\mathscr O(T_\alpha)$, spanned by solutions of the system. Any non-empty simply connected pairwise intersection $T_{\alpha\beta}$ defines the canonical isomorphism between $\mathscr L_\alpha$ and $\mathscr L_\beta$, in other words, the spaces $\mathscr L_\alpha$ are naturally globally organized into the \emph{holomorphic vector bundle} over $\C P^1\ssm \S$. If the initial system is regular, then one can extend this bundle in an analytic way over the singular locus as explained in \cite[Proposition 18.8]{thebook}. Moreover, the procedure of analytic continuation induces on this bundle the additional structure of a connection in such a way that the functions from $\mathscr L_\alpha$ correspond to horizontal sections of this connection over $T_\alpha$. By construction, this connection is holomorphic and flat over $\C P^1\ssm\S$. At the singular points the connection has an ``atomic'' curvature defined by the local monodromy operators. It is this geometric structure which naturally appears in many applications, which makes regular systems so ubiquitous in mathematics.
\end{Rem}

\subsection{Root counting for Fuchsian systems}
The fact that the linear span of all entries of $X$ is well defined, allows us to formulate the \emph{root counting problem} for the system \eqref{pf1}. In what follows we use the symbol $\#A$ for the number of isolated points in the set $A$. Consider the supremum
\begin{equation}\label{count}
    \mathcal N(\Omega)=\sup_{\displaystyle T\cap\S=\varnothing}~~~\sup_{\displaystyle c_{ij}\in\C}\#\biggl\{p\in T:\sum_{i,j=1}^n c_{ij}X_{ij}(t)\text{ has an isolated root at }p\biggr\}.
\end{equation}
Here $T\subset\C P^1\ssm\S$ denotes an open triangle (an area bounded by 3 circular arcs or line segments) free from singularities of the differential equation \eqref{pf1} and supremum is taken over all such triangles free from singular points and all linear combinations $\sum c_{ij}X_{ij}(t)$. If this supremum is finite, then we say that the system \eqref{pf1} \emph{admits an upper bound for the number of isolated zeros of solutions}, or, simply, \emph{admits root counting}.

\begin{Ex}
A system with an irregular (wild) singular point $t=a$ in general does not admit root counting. In fact, by general arguments from the theory of value distribution, any solution of super-polynomial growth admits almost every value infinitely often, in any sufficiently large sector around the singular point.
\end{Ex}

For Fuchsian singularities the possibility of \emph{local} root counting (say, in a small circular neighborhood of the singular point) depends on the spectrum of the residue of the matrix 1-form $\Omega$.

\begin{Ex}[continuation of Example~\ref{ex:euler}]
It is well known that if all the eigenvalues $\l_j$ of the residue matrix $A=A_0=-A_\infty$ are real, then any such linear combination with real coefficients $c_j$ may have no more than $n-1$ positive real roots (the non-oscillation, or the Chebyshev property): one can prove this fact using the Rolle lemma on alternation between zeros of a real smooth function and its derivative. The complex roots of the linear combinations with complex coefficients can also be counted using a complex version of the Rolle lemma, as shown in \cite{jdcs-96}. In particular, under the same assumption $\l_1,\dots,\l_n\in\R$, the Euler system admits an upper bound for the number of roots by the expression $(n-1)+2\pi\cdot|\l_1-\l_n|$ (assuming the eigenvalues are labeled in a monotone order).

On the other end, if $n=2$ and $\l_1=-\l_2=\sqrt{-1}$ is a pair of complex conjugate eigenvalues, then the function $\frac12(t^{\mathrm i}+t^{-\mathrm i})=\cos\ln t$ has infinitely many positive isolated roots accumulating to both $t=0$ and $t=+\infty$. In this case $\mathcal N=+\infty$ and the system does not admit root counting.
\end{Ex}

In fact, the property that is essential for the root counting, is more naturally expressed in terms of the \emph{local monodromy}, i.e., the monodromy operators associated with the \emph{small loops} (around different singular points). For a Fuchsian singularity the eigenvalues $\nu_1,\dots,\nu_j$ of the monodromy of a small loop are exponents, $\nu_j=\exp 2\pi i \l_j$, of eigenvalues $\l_1,\dots,\l_n$ of the residue matrix \cite[Corollary 16.20]{thebook}. Thus all eigenvalues of the residue are real if and only if all eigenvalues of the local monodromy operators have unit modulus, $|\nu_j|=1$. (Note that this description fails for regular non-Fuchsian singularities).

\begin{Def}\label{def:qupotent}
We say that a regular system \eqref{pf1} has \emph{quasiunipotent monodromy} (or simply \emph{is quasiunipotent}), if the  monodromy operators $M_j$ of all small loops around distinct singular points $a_j\in\S$ are roots of unity (in particular, have only modulus 1 eigenvalues).
\end{Def}

Using the technique developed by A.~Gabrielov and A.~Khovanskii, one can prove the following (relatively easy) theorem, whose assumptions are algebraic and can be effectively verified.

\begin{Thm}\label{thm:exist}
A regular Pfaffian system \eqref{pf1} with quasiunipotent local monodromy admits finite root count, $\mathcal N(\Omega)<+\infty$.
\end{Thm}

Unfortunately, this approach, albeit very general and powerful, does not allow to make the bound for the counting function explicit, even for Fuchsian systems.

\subsection{Explicit bounds for Fuchsian systems}\label{sec:toward}
Several simple examples show that although each quasiunipotent Fuchsian system \eqref{fs} with the Pfaffian matrix $\Omega$ having only simple poles with real spectrum residues, admits a finite bound $\mathcal N(\Omega)$, this bound cannot be uniform on all systems of the given degree $m$ (the number of singular points) and dimension $n$.

\begin{Ex}\label{ex:growth}
The diagonal Euler $(2\times 2)$-system with two integer eigenvalues $0$ and $\ell\in\mathbb N$ admits the complex polynomial $t^\ell-1$ as a solution; this polynomial has as many as $\ell$ isolated roots in a suitable triangle $T\subseteq\C P^1$. This suggests that the bound $\mathcal N(\Omega)$ may well grow to infinity as the spectral radii of the residues grow to infinity.
\end{Ex}

A slightly less obvious obstruction to uniformity of the bound is the sensitivity of the quasiunipotence condition to the ``collision'' of singularities, when the distance $|a_i-a_j|$ between two distinct singular points $a_i,a_j\in\S$ tends to zero. It may well happen \cite{annalif-09} that the monodromy along each small loop is quasiunipotent, but the loop which encircles both singular points, has non-quasiunipotent monodromy. Thus in the limit when $a_i=a_j$, the system may have infinitely many roots of solutions, and an arbitrarily large finite number of them would persist for $0<|a_i-a_j|\ll 1$.

These phenomena can be re-stated as follows. The collection $\mathcal F_{n,m}$ of all Fuchsian systems \eqref{fs} of dimension $n$ with $m$ singular points is a semialgebraic variety, a subset of a complex affine space $\C^N$ which can be represented as the difference of two affine algebraic varieties.  For instance, consider all possible residue matrices $A_1,\dots,A_m\in\operatorname{Mat}(n,\C)\simeq\C^{n^2}$ and all possible pole assignments $a_1,\dots,a_m\in\C$, and exclude the ``degenerate cases'' when $A_i=0$ (disappearance of one of the singular points) or $a_i=a_j$ (collision of the poles). Then the counting function $\mathcal N(\cdot)$ becomes well defined on a proper semialgebraic subset  $\mathcal S_{n,m}\subsetneq\mathcal F_{n,m}$ defined by the condition $\operatorname{Spec}(A_j)\subseteq\R$, $j=1,\dots,m$. The counting function $\mathcal N$ is lower semicontinuous (all isolated complex roots of linear combinations survive sufficiently small perturbations if counted in open triangles with proper multiplicities)  and hence this function is bounded on any compact subset of $\mathcal S_{n,m}$. However, since $\mathcal S_{n,m}$ itself is non-compact, $\mathcal N(\cdot)$ can apriori be unbounded, and the above examples show that indeed $\mathcal N$ tends to infinity at least near some parts of the boundary\footnote{More precisely, we have to consider the projective boundary for which the closure should be taken with respect to the projective space which compactifies the affine space $\C^N$.} $\partial\mathcal S_{n,m}=\overline{\mathcal S_{n,m}}\ssm\mathcal S_{n,m}$.

The main result of the paper \cite{annalif-09} is the claim that the counting function $\mathcal N$ grows no faster than polynomially near the boundary $\partial \mathcal S_{n,m}$. This means that for any semialgebraic ``reciprocal distance to the boundary'' $\rho(\Omega)=1/\operatorname{dist}(\Omega,\partial \mathcal S_{n,m})$ one has an estimate of the form $\mathcal N(\Omega)\le \rho^k(\Omega)$, where $k<+\infty$ is a finite power which depends, of course, on the choice of $\rho$ (and also implicitly on $n$ and $m$). Since by the \polishL ojasiewicz inequality any two such distances $\rho,\rho'$ are equivalent, $\rho'\le \rho^{O(1)}$ and vice versa, the above polynomial growth statement is in fact independent of the choice of the distance $\rho$ as long as the latter remains semialgebraic.

However, the real strength of the bound achieved in \cite{annalif-09} is its constructive nature: it allows to give an explicit upper bound for $k$ in terms of $n,m$, assuming any given distance function. For instance, for a Pfaffian system \eqref{pf1} with residue matrices $A_j$ and the singularities $a_j$ (in a given affine chart) we can define
\begin{equation}\label{cover}
    \rho(\Omega)=2+\sum_j |A_j|+\sum_{i\ne j}\frac1{|a_i-a_j|}<+\infty.
\end{equation}
(In fact, this function does not ``notice'' the disappearance of the singular points when $A_j\to0$; the constant 2 is added in order to avoid special treatment of systems with $\rho\le 1$, as we are interested only in large values of $\rho$).

\begin{Thm}[\cite{annalif-09}]
\begin{equation}\label{double-fuchs}
 \mathcal N(\Omega)\le \rho(\Omega)^{\textstyle 2^{\scriptstyle\Poly(n,m)}},
\end{equation}
where $\Poly(n,m)$ is an explicit polynomial expression in $n,m$ of degree at most $1+\dim\mathcal S_{n,m}=1+mn^2+m$, where $\dim\mathcal S_{n,m}$ is the dimension of the semialgebraic set $\mathcal S_{n,m}$.
\end{Thm}

This double exponential bound may seem too excessive, yet one can argue following \cite{annalif-09} that a \emph{non-uniform} bound polynomially growing near the boundary of a semialgebraic set in a high-dimensional affine space and invariant by the choice of the semialgebraic distance, should necessarily be growing at a comparable rate as the dimension of the space grows to infinity.

\subsection{Towards uniform bounds: obstructions and solutions}
As was already mentioned in the previous section, one cannot expect to obtain a bound which would be uniform over all domain of finiteness of the counting function $\mathcal N$. Thus one has to focus on constructing some smaller ``naturally noncompact'' subsets of $\mathcal S_{n,m}$, for which uniform bounds can be achieved.

The weakest (also the most important) barrier occurs on the diagonal set $\bigcup_{i\ne j}\{a_i=a_j\}$. As was explained above, the obstruction to existence of a uniform bound is hidden in the fact that after collision of two singular points with quasiunipotent local monodromy one can obtain a point with non-quasiunipotent monodromy, to which infinitely many roots may accumulate. To prevent this, one could impose the quasiunipotence condition also on the monodromy for the loop encompassing both colliding singularities. However, this condition is \emph{non-algebraic} and in general non-verifiable: unlike the local computation with finite order jets which allows to compute the local monodromy operator \cite[Corollary 16.19]{thebook}, the monodromy along a large loop requires in general knowing exact solutions which are generically non-algebraic.

Instead one can impose an additional condition that all systems from the conjectural class are \emph{isomonodromic} to each other. This condition stipulates that for each closed loop avoiding the singular locus of a given system, all sufficiently close systems from this class have the same monodromy along this loop modulo conjugacy. In particular, this condition implies that all residue matrices remain in the same conjugacy class and hence their characteristic polynomials are constant.

The isomonodromy condition is surprisingly rigid: in general, it (locally) uniquely determines the residue matrices as functions of the polar configuration. More precisely, in the so called non-resonant case\footnote{The non-resonance assumption requires that the difference between any two eigenvalues of each residue matrix $A_j$ is non-integer.} the residues $A_i=A_i(a_1,\dots,a_n)$ of the Fuchsian system \eqref{fs} considered as functions of the pole location $a_1,\dots,a_n$, satisfy the integrable system of (nonlinear) matrix Pfaffian equations, called the \emph{Schlesinger system}:
\begin{equation*}
    \d A_i+\sum_{i\ne j}[A_i,A_j]\,\frac{\d (a_i-a_j)}{a_i-a_j}=0,\qquad i=1,\dots,n.
\end{equation*}
Solutions of this system may exhibit singularities both on the diagonal (``collision locus'' $\bigcup_{i\ne j}\{a_i-a_j=0\}$) and elsewhere (depending on the initial conditions), since the system is nonlinear (quadratic). Therefore in general the bound \eqref{double-fuchs} will diverge because of the ``explosion'' of the residues in the isomonodromic collision of Fuchsian singularities.

This phenomenon suggests that parametrization of isomonodromic families of Fuchsian systems by the location of their poles and the corresponding residues leads to creation of singularities on the diagonal (the collision locus). The alternative is to change the point of view and consider parametric classes of regular linear systems as (Pfaffian) systems on multidimensional phase space. Then the Schlesinger equations can be interpreted as the holonomy (flatness) condition equivalent to the local existence of solutions. This point of view is formalized as follows.

Consider a projective space $\C P^m$ and a rational $(n\times n)$-matrix Pfaffian form $\Omega=\{\Omega_{ij}\}_{i,j=1}^n$ on it. Denote by $\S\subseteq\C P^m$ the polar locus of $\Omega$, the union of polar loci of all scalar forms $\Omega_{ij}\in\bigland^1_\S(\C P^m)$. Outside of $\S$ we can consider the system formally defined by the same equations as \eqref{pf1}, yet this time with several ``independent variables'',
\begin{equation}\label{pfm}
    \d X=\Omega X,\qquad \Omega\in\operatorname{Mat}_{n\times n}\bigl(\bigland^1_\S(\C P^m)\bigr),
    \quad X\in\operatorname{GL}\bigl(n,\mathscr O(\C P^m\ssm\S)\bigr),
\end{equation}
which becomes a system of partial (rather than ordinary) differential equations in an affine chart on $\C P^m$. Obviously, we are interested only in systems which locally admit a holomorphic fundamental solution matrix $X$ so that $\Omega=\d X\cdot X^{-1}$. Applying the exterior derivative to both parts of this formula, we conclude that
\begin{equation}\label{flatness}
 \d \Omega=-\d X\land \d(X^{-1})=-\d X\land (-X^{-1}\,\d X\cdot X^{-1})=(\d X\cdot X^{-1})\land(\d X\cdot X^{-1})=\Omega\land\Omega.
\end{equation}
This \emph{local integrability condition} is also sufficient for the local existence of solutions; it has a geometric meaning that the connection defined by the form $\Omega$ on the trivial bundle over $\C P^m$ is \emph{flat}, i.e., has zero curvature.

In the same way as in one dimension, we define the \emph{regularity condition} for the system \eqref{pfm} on the singular locus, by requiring that any solution grows no faster than polynomially when approaching $\S$ inside any simply connected semialgebraic set $U\subseteq\C P^m\ssm\S$ (the semialgebraicity condition is required to exclude sets spiraling around $\S$):
\begin{equation}\label{regular}
    |X(z)|+|X^{-1}(z)|\le C\operatorname{dist}(z,\S)^{-N}\qquad\text{for some finite }N<+\infty,\quad z\in U.
\end{equation}

The root counting problem for solutions of the system \eqref{pfm} (defined on $\C P^m$) by definition reduces to that for the system \eqref{pf1} (defined on $\C P^1$) by considering the restriction on all possible projective lines not lying entirely in the singular locus: formally we can recycle the definition \eqref{count} by adding that the supremum is taken over all triangles $T$ belonging to all projective lines in $\C P^1$.

The next section contains the sufficient conditions for a regular integrable system \eqref{pfm} on $\C P^m$ to admit finite root count. We formulate a proper multidimensional generalization for the quasiunipotence condition on the local monodromy operators of a regular system \eqref{pf1}, which guarantees (together with the above conditions of integrability and regularity) the finiteness of roots in any triangle.

\subsection{The main result on quasiunipotent systems}
\begin{Def}
A \emph{small loop} around the polar divisor $\S$, centered at a point $a\in\S$, is the free homotopy class of the image of any sufficiently small circle $\{|t|=\e\}$ by a holomorphic map $\gamma:(\C,0)\to(\C P^m, \S)$ whose image does not belong to $\S$.
\end{Def}
If $a\in\S$ is a smooth point of $\S$, i.e., $\S$ is locally near $a$ represented by a single equation $\{f=0\}$ with $f(a)=0$, $\d f(a)\ne 0$, then all small loops centered at $a$ are free homotopy equivalent to each other. On the other hand, if $a$ is a singular point of the locus $\S$, this may not be the case anymore. However, certain spectral properties of the corresponding monodromy operators are inherited from generic to degenerate small loops.

\begin{Thm}[Kashiwara \cite{kashiwara}]
Assume a flat regular system \eqref{pfm} with the polar locus $\S$ possesses the following property: its monodromy operator along any small loop centered at a smooth point $\S$ is quasiunipotent.

Then the same quasiunipotence hold for \emph{any} small loop, including those centered at the singular points of $\S$ as well.
\end{Thm}

This theorem is in fact a result on the local topology of the complements to embedded analytic hypersurfaces. Note that all small loops centered at smooth points of $\S$ are also free homotopic to each other in the case when $\S$ is an \emph{irreducible} algebraic subvariety in $\C P^m$. Thus the quasiunipotence condition is sufficient to verify at a finite number of points (one smooth point for each irreducible component of $\S$). Moreover, if $\Omega$ has a first order pole on $\S$ (i.e., $f\Omega$ is holomorphically extendable for any local equation of $\S$), then one can correctly define the residue as a holomorphic matrix function on the smooth part of $\S$, such that all values of the residue are conjugate to each other along any irreducible component of $\S$. For such systems the quasiunipotence condition is equivalent to the assumption that the residue matrix has only rational eigenvalues.

A system \eqref{pfm} meeting the assumptions of the Kashiwara theorem, will be called \emph{quasiunipotent}, naturally extending thus Definition~\ref{def:qupotent}. Not surprisingly, after the notions of regularity \eqref{regular}, quasiunipotence and the root counting function $\mathcal N$ are generalized from the one-dimensional to the multidimensional case, Theorem~\ref{thm:exist} remains valid.

\begin{Thm}\label{thm:existm}
A \emph{quasiunipotent} flat regular Pfaffian system \eqref{pfm} admits finite root count: $\mathcal N(\Omega)<+\infty$.
\end{Thm}

As its one-dimensional version, the mere finiteness of the bound can be proved using very general arguments of Gabrielov--Khovanskii--Varchenko type. To make this bound explicit, a completely new approach is required. Even to formulate such a result, one has to introduce an additional characteristic of linear systems \eqref{pfm} which would essentially restrict the ``magnitude of the coefficients'' to exclude the situation described in Example~\ref{ex:growth}. Such a characteristic must necessarily enter the expression of any explicit bound, and the discussion in \secref{sec:toward} suggests that the norm of the residue matrices is a wrong choice.

\begin{Def}
We say that a Pfaffian system \eqref{pfm} is \emph{defined over $\Q$}, if in some affine chart $\C^m=\{t_1,\dots,t_m\}$ on $\C P^m$ all rational 1-forms $\Omega_{ij}$ have coefficients from the field $\Q(t_1,\dots,t_m)$, i.e., ratios of polynomials with integer coefficients from $\mathbb Z[t_1,\dots,t_m]$:
$$
 \Omega=\{\Omega_{ij}\}_{i,j=1}^m,\qquad
 \Omega_{ij}=\sum_{k=1}^m \frac{P_{ijk}(t)}{Q_{ijk}(t)}\,\d t_k,\qquad P_{ijk},Q_{ijk}\in\mathbb Z[t]=\mathbb Z[t_1,\dots,t_m].
$$
The largest integer number which is used to write explicitly the polynomials $P_{ijk},Q_{ijk}$, is called the \emph{complexity} of a system defined over $\Q$.
\end{Def}

This definition clearly depends on the choice of the representation of $\Omega$ (e.g., it changes after the affine change of variables $t\mapsto 2t$), yet for our purposes it will always be used in the sense that at least one representation of the given complexity exists. Having an upper bound on the complexity implies that any \emph{numeric} characteristic of the system, that can be obtained by a finite number of algebraic manipulations, can be explicitly majorized in the sense of the absolute value. This allows to translate any explicit algebraic algorithm into an explicit bound for the complexity of its results. In particular, we obtain the following bound for the root counting problem.

\begin{Thm}[\cite{invmath-10}]\label{thm:main}
If a quasiunipotent flat regular Pfaffian system \eqref{pfm} is defined over $\Q$ and has complexity at most $s\in\mathbb N$, $s\ge 2$, then
\begin{equation*}
    \mathcal N(\Omega)\le s^{2^{\scriptstyle\mathrm{Poly}(d,n,m)}},\qquad \mathrm{Poly}(d,n,m)\le O\bigl((d^5m^5n^{20})\bigr).
\end{equation*}
Here $d=\max_{i,j,k}(\deg P_{ijk},\deg Q_{ijk})$ is the degree of the Pfaffian form, $m=\dim_\C \C P^m$ the dimension of the space \textup(number of the independent variables\textup) and $n=\dim\Omega$ the dimension of the system \textup(number of the dependent variables\textup). The polynomial of degree $30$ occurring in the exponent is explicit and can be computed.
\end{Thm}

This theorem implicitly introduces a class of special ``quasialgebraic'' multivalued functions, which generalizes the class of algebraic functions in the sense that they admit finite root count of any branch.

\begin{Def}\label{def:Q}
A \emph{Q-system} is a flat regular quasiunipotent Pfaffian system defined over $\Q$. A \emph{Q-function} is a multivalued analytic function which is a component (or linear combination of components) of a matrix solution of a Q-system.
\end{Def}

The assertion that a multivalued nondegenerate matrix function $X(t)$, $t\in\C P^m\ssm\S$, is a Q-function, consists therefore of four separate statements:
\begin{enumerate}
 \item $X$ undergoes a linear monodromy transformation $X\mapsto XM$, $M\in\textrm{GL}(n,\C)$, when continued along a loop $\gamma$ avoiding the ramification locus $\S$;
 \item This monodromy transformation has only roots of unity as the eigenvalues, if $\gamma$ is a small loop (by the Kashiwara theorem, this condition needs to be checked only for small loops around smooth part of $\S$, one loop for each irreducible component of $\S$ suffices);
 \item $X$ is regular, i.e., exhibits at most moderate growth on $\S$ (in particular, $X$ has genuine finite order poles along the components of $\S$ with the identical monodromy or satisfies the growth control estimate \eqref{regular});
 \item The logarithmic derivative $\Omega=\d X\cdot X^{-1}$, necessarily rational in $t$ under the above assumptions, is defined over $\Q$.
\end{enumerate}
The flatness condition is automatically satisfied.

\section{Q-systems, Q-functions and their properties}

\subsection{Elementary examples of Q-functions}
Polynomials $\C[t_1,\dots,t_m]$ of a given bounded degree $d$ are a trivial example of $Q$-functions. Indeed, the exterior derivative of each monomial $t^\nu=t_1^{\nu_1}\cdots t_m^{\nu_m}$ is a linear combination of monomial 1-forms of smaller degree with natural coefficients not exceeding $d$, which means that the corresponding logarithmic derivative in this affine chart is constant and defined over $\mathbb Z$. The growth condition is obviously satisfied and the monodromy is identical, hence trivially quasiunipotent. Of course, the double exponential root count provided by Theorem~\ref{thm:main}, is enormously excessive in this case.

\begin{Prop}\label{prop:alg-pf}
Any algebraic function defined over $\Q$ is a Q-function.
\end{Prop}

Before proving the Proposition, we need a simple algebraic fact which is a cornerstone of the Elimination theory.

\begin{Lem}\label{lem:elim}
Let $S\subset\C^n\times\C$ be an algebraic hypersurface defined by the polynomial equation $S(t,x)=0$ of degree $d=\deg_x S$, with $S$ defined over $\Q$. Then any rational function $R(t,x)=A(t,x)/B(t,x)\in\Q(t,x)$ restricted on $S$ coincides with a fraction of the special form,
\begin{equation*}
    R(t,x)=\frac{U(t,x)}{Q(t)},\qquad U\in\Z[t,x],\ Q\in\Z[t],\quad\deg_x U\le d-1.
\end{equation*}
The new denominator $Q$ depends only on the polynomial $S$ defining the surface and the denominator $B$ of the rational function $R$. The complexity of this representation can also be explicitly controlled in terms of the complexities of $R,S$ and their degrees.
\end{Lem}

\begin{proof}[Proof of the Lemma]
Consider the resultant of $S$ (the equation) and $B$ (the denominator) considered as polynomials in $x$ with coefficients in the field $\Q(t)$. By definition, this is the element of the field $\Q(t)$ representable as the linear combination,
\begin{equation*}
    Q(t)=a(t,x)S(t,x)+b(t,x)B(t,x),\qquad a,b\in\Q(t,x),
\end{equation*}
with $\deg_x a<\deg_x B$, $deg_x b< \deg_x S$. The resultant $Q(t)$ vanishes if and only if $S(t,\cdot)$ and $B(t,\cdot)$ have a common $x$-root.

Evaluating this identity on the hypersurface $S$, we conclude that $1/B$ equals the polynomial $b/Q$. Dividing $A(t,x)b(t,x)$ with remainder by $S(t,x)$ (using the Euclid algorithm) in the ring $\Q(x)[t]$, we find $U,V\in\Q(x)[t]$ such that $Ab=VS+U$ with the required bound for the degree $\deg_x U\le\deg_x S$. Multiplying the numerator and denominator by a suitable term from $\Z[x]$, we can guarantee that the ratio involves polynomials in $t,x$ with integer coefficients.
\end{proof}

\begin{proof}[Proof of the Proposition]
The algebraic equation $P(t_1,\dots,t_n,y)=0$ defines an algebraic surface in $\C^n\times\C$, and its exterior derivative $\d P=0$ implies the Pfaffian equation
\begin{equation*}
    \d y=-\frac1{P'(t,y)}\sum_{k=1}^n P_k\,\d t_k,\qquad P'=\frac{\partial P}{\partial y},\quad P_k=\frac{\partial P}{\partial t_k},
    \quad P',P_k\in\Q[t,y].
\end{equation*}
Applying Lemma~\ref{lem:elim}, we can replace the restriction of this equation on $P=0$ by a rational Pfaffian equation whose denominator $\Delta(x)$, the resultant of $P(t,y)$ and $P'(t,y)$ as polynomials in $y$ (the discriminant of $P(t,\cdot)$), depends only on $t$ and the numerator has degree in $y$ not exceeding $d-1$:
\begin{equation}\label{ode-alg}
    \d y=\frac1{\Delta(t)}\sum_{k=1}^n U_k(t,y)\,\d t_k,\qquad U_k\in\Z[t,y],\qquad \deg_y U_k\le d-1.
\end{equation}
The same procedure can be applied to the Pfaffian equation \eqref{ode-alg} multiplied by any power $y^{j-1}$: the result will be a rational Pfaffian form for the exterior derivative $\d (y^j)$, $j=2,3,\dots$ whose denominator is still the same discriminant $\Delta(t)$ and the numerators are polynomials of degree $\le d-1$ in  $y$.

In other words, we obtain a \emph{linear} system of Pfaffian equations
$$
 \d x_j=\frac1{\Delta(t)}\sum_{k=1}^{d}\sum_{l=1}^n u_{jkl}(t) x_k\,\d t_l
$$
for the variables $x_1,\dots,x_n$, $x_j=y^{j-1}$, $j=1,\dots,n$, with coefficients from $\Q(t)$ (obviously, $\d x_1=0$). The algorithm is explicit and allows for estimation of the corresponding complexity.

The monodromy of this linear systems simply permutes between themselves the branches of the algebraic function. Any such transformation is cyclic (its Jordan blocks are cycles of period no greater than $n$) hence the corresponding monodromy is quasiunipotent for \emph{any} loop.
\end{proof}

\begin{Cor}
The \emph{general algebraic function} defined by a ``universal polynomial equation of degree $n$'',
\begin{equation}\label{upe}
    t_0\,y^n+t_1\,y^{n-1}+\cdots+t_{n-1}\,y+t_n=0,
\end{equation}
is a Q-function of its arguments $t_0,\dots,t_n$ considered as homogeneous coordinates on the projective space $\C P^n$. \sq
\end{Cor}

\subsection{Special functions}
Hypergeometric and Riemann P-functions also are Q-functions for the rational values of the parameters almost by definition: the second order linear equation
\begin{equation*}
    t(t-1)\cdot\frac{\d^2 y}{\d t^2}+[c-(a+b+1)t]\cdot\frac{\d y}{\d t}-ab\cdot y=0
\end{equation*}
has three regular singular points at $t=0,1,\infty$ and the corresponding residues are rational for rational $a,b,c\in\Q$. Thus a vast number of special functions turns out to be Q-functions and admit explicit bounds on the total number of roots (depending on $a,b,c$).

However, as a function of all four variables $t,a,b,c$, the Riemann P-function \emph{is not a Q-function}: the parametric family above {is not} isomonodromic and hence cannot be transformed to an integrable Pfaffian system on $\C P^4$. Besides, some solutions grow faster than polynomially as $(a,b,c)\to\infty$.

\subsection{Periods as Q-functions}
The first general example of transcendental Q-functions is provided by Abelian integrals
\begin{Def}
An \emph{Abelian integral} is the period $I(\G,\omega)$ of a rational 1-form $\omega\in\bigland^1(\C P^2)$ over an algebraic cycle on an algebraic curve $\G\subset\C P^2$, considered as the function of all relevant data.
\end{Def}
Recall that an algebraic cycle is an element $c$ of the homology group $H_1(\G,\mathbb Z)$ of an algebraic curve $\G\subseteq\C P^2$. If the curve is nonsingular, then this cycle (as a locally constant section of the Gauss--Manin connection) continuously depends on the curve $\G$ in a natural way. In an affine chart $(x,y)$ on $\C P^2$ an Abelian integral takes the form
\begin{equation}\label{ai-gen}
    I(H,Z,P,Q)=\oint_{c\subseteq\{H(x,y)=0\}}\frac{P(x,y)\,\d x+Q(x,y)\,\d y}{Z(x,y)},
    \qquad H,Z,P,Q\in\C[x,y].
\end{equation}
The function $I$ is naturally defined on an the Zariski open subset of nonsingular algebraic curves of a given degree $n$, parameterized by points of the suitable projective space (coefficients of the polynomial $H$) and the projective space of rational forms of a given degree $\le m$.

\begin{Thm}\label{thm:Q-pic-fuchs}
The Abelian integral \eqref{ai-gen} is a Q-function of its arguments. The dimensions, degree and complexity of this Q-function are bounded by a polynomial in the degrees of $P,Q,Z,H$.
\end{Thm}

\begin{Rem}
The fact that the Abelian integrals are regular functions satisfying some linear (Picard--Fuchs) ordinary differential equations with rational coefficients, is well known under much more general assumptions. Consider an arbitrary surjective polynomial map $F:\C^n\to\C^m$, $n\ge m$, denote by $\G_t=F^{-1}(t)$ the fibers of this map, which are affine algebraic varieties. For a generic value of $t$ (i.e., off an algebraic surface $\S\subset\C ^m$) $\G_t$ is nonsingular variety, whose homology is generated by algebraic cycles of dimension $k=n-m$ which locally continuously depend on $t\notin\S$. By a theorem of Grothendieck \cite{groth}, the dual module (cohomology) is generated by the same number of polynomial $k$-forms whose restrictions on $\G_t$ are linear independent outside a proper polynomial hypersurface (we can still denote it by $\S$). Thus the period matrix $X(t)$, whose entries are all periods of the corresponding forms, is well-defined, holomorphic and nondegenerate outside $\S$ matrix function, which grows moderately as $t\to\S$ and is monodromic: after analytic continuation along any loop $\gamma$ around $\S$, the matrix function acquires a right (constant) matrix factor $M_\gamma$ completely determined by the topology of $F$ considered as a bundle. The quasiunipotence of $M_\gamma$ is also a well-known fact. The matrix logarithmic derivative $\Omega=\d X\cdot X^{-1}$ is a rational matrix 1-form on $\C^m$ (hence extends on $\C P^m$) with poles on $\S$, and the first three conditions from Definition~\ref{def:Q} are verified.

However, this quite general argument does not allow to see why the matrix $\Omega$ is defined over $\Q$, and does not allow to estimate its complexity. One has to construct explicitly the matrix $\Omega$ in such a way that its complexity over $\Q$ can be bounded.

This construction was produced in \cite{invmath-10} for the particular case where the polar locus $\{Z=0\}$ of the forms is a fixed line in $\C P^2$ so that the space of all forms is a linear space of polynomial 1-forms in a fixed affine chart. The general case will be treated elsewhere.
\end{Rem}

\subsection{Elementary constructions with Q-functions}
The ring operations (addition and multiplication) preserve the class of Q-functions, affecting only their natural characteristics (degree, dimension, number of variables and complexity).

Indeed, for two Q-systems defined on two different projective spaces by two systems of quasiunipotent flat regular Pfaffian equations
\begin{equation*}
    \d X=\Omega X,\qquad \d Y=\Theta Y,
\end{equation*}
of the sizes $n\times n$ and $m\times m$ respectively, one can form the block diagonal system of size $(n+m)\times(n+m)$,
\begin{equation*}
    \d \begin{pmatrix} X \\ Y \end{pmatrix}=\begin{pmatrix} \Omega&\\&\Theta\end{pmatrix}\cdot\begin{pmatrix}X\\Y\end{pmatrix}
\end{equation*}
which is obviously flat, regular and quasiunipotent (the monodromy group is the Cartesian product of the two individual groups). This system (whose matrix can be denoted by $\Omega\oplus\Theta$) is satisfied by the direct sum $X\oplus Y$ of the corresponding matrix solutions.

The tensor product $X\otimes Y$, also satisfies the linear system of the size $nm\times nm$. To see this, it is more convenient to look at the two systems in the vector (rather than matrix) forms:
\begin{equation*}
    \d x=\Omega x,\quad \d y=\Omega y,\qquad x=(x_1,\dots,x_n),\quad y=(y_1,\dots,y_m).
\end{equation*}
Then the $n\times m$-vector with the coordinates $z_{ij}=x_i y_j$ satisfies the system of linear Pfaffian equations
\begin{equation*}
    \d z_{ij}=\d x_i\cdot y_j+x_i\cdot\d y_j=\sum_{k=1}^n\Omega_{ik}z_{kj}+\sum_{l=1}^m \Theta_{jl}z_{il}. 
\end{equation*}
The matrix elements of the corresponding $nm$-matrix 1-form are certain sums of the matrix elements of the forms $\Omega$ and $\Theta$ respectively. We abbreviate the corresponding computation using the tensor product notation as follows\footnote{In the language of connections one sometimes uses the following notation for the tensor product, $$(\nabla'\otimes\nabla'')(X\otimes Y)=(\nabla'X)\otimes Y+X\otimes(\nabla''Y),\quad \nabla'=\d-\Omega,\ \nabla''=\d-\Theta.$$ This may justify the notation $\Omega\otimes\Theta$  for the matrix form corresponding to the tensor product $\nabla'\otimes\nabla''=\d-\Omega\otimes\Theta$ of the two connections.},
\begin{equation*}
    \d (X\otimes Y)=\d X\otimes Y+X\otimes \d Y=(\Omega\otimes I+I\otimes\Theta)(X\otimes Y).
\end{equation*}
One can easily verify that all characteristic properties of the Q-systems are preserved by the tensor product, with explicit control over the complexity growth.

These elementary properties imply that the extension $\Bbbk(X)=\C(t,X)=\C(t_1,\dots,t_m,\allowbreak X_{11},\dots,X_{nn})$ of the differential field $\Bbbk=\C(t_1,\dots,t_m)$ by Q-functions $X=X(t)$ is a differential field which admits an explicit upper bound for the number of isolated roots of each its element.

\begin{Thm}\label{thm:root-ext}
Let $X$ be any fundamental matrix solution of a fixed Q-system \eqref{pfm} and $\Bbbk(X)$ the differential extension field.

Then for any element $\f\in\Bbbk(X)$ and any one-dimensional triangle $T$ entirely belonging to the domain of analyticity of $\f$, the number of isolated zeros of $\f$ in $T$ is bounded by an explicit function of the degree $\deg_{t,X}\f:$
\begin{equation}\label{field-ext}
    \forall T\subset\C P^m\ssm\S,\ \forall\f\in\Bbbk(X)\qquad \#\{t\in T:\f(t)=0\}\le 2^{2^{\operatorname{Poly(\delta)}}},\qquad \delta=\deg_{t,X}\f.
\end{equation}
The polynomial $\operatorname{Poly}(\delta)$ explicitly depends on the Q-function $X$ via the dimension, degree and complexity of the latter.
\end{Thm}

\begin{proof}
All monomials of the form $t^\alpha X^\beta$ with $|\alpha|+|\beta|\le \delta$ (in the standard multi-index notation) satisfy a large Q-system whose dimension, degree and complexity are bounded by some explicit polynomial expressions in the dimension, degree and complexity of the Q-function $X$, and the degree $\delta$. The rest follows from the Main Theorem~\ref{thm:main}.
\end{proof}

\subsection{Transformations of Q-functions by rational and algebraic maps}
If $X$ is a matrix Q-function on $\C P^m$ and $f:\C P^l\to\C P^m$ a rational map defined over $\Q$, then the natural pull-back $f^*X$ will be a Q-function on $\C P^l$, defined by the Pfaffian system with the matrix 1-form $f^*\Omega$.

Indeed, the conditions of flatness and regularity are obviously satisfied. The pullback matrix 1-form $f^*\Omega$ is clearly defined over $\Q$. The only condition that has to be verified is that on the monodromy. But the image of a small loop by a rational map is at worst a \emph{multiple} of a small loop, hence the eigenvalues of the corresponding monodromy are powers of the initial eigenvalues. 

Yet a much more interesting question is to find out under what conditions the class of Q-functions is closed by composition, in particular, when the substitution of the independent variables $t=F(s)$ with algebraic functions $F$ transforms Q-functions again to Q-functions. Among other things, this would allow to replace ``line triangles'' in the formulation of the counting problem, by arbitrary semialgebraic triangles.

\begin{Thm}\label{thm:Q-alg}
Consider a Q-system \eqref{pfm} and a transformation defined by an algebraic change of independent variables,
\begin{equation}\label{alg-tr}
    t_j=t_j(s_1,\dots,s_l),\ j=1,\dots,m, \qquad P_j(t_j,s_1,\dots,s_l)=0,\ P_j\in\Q[t_j,s].
\end{equation}
Then the multivalued matrix function $X(t(s))$ is a Q-function of the new variables.
\end{Thm}

\begin{proof}
Consider the Pfaffian system \eqref{pfm} in the affine chart $(t_1,\dots,t_m)$:
\begin{equation*}
    \d x_i=\sum  R_{ijk}(t)x_j\,\d t_k,\qquad  R_{ijk}\in\Q(t_1,\dots,t_m),\ \deg R_{ijk}\le d.
\end{equation*}
Differentiating the identities \eqref{alg-tr}, the differentials $\d t_j$ can be replaced by rational forms,
\begin{equation}\label{dtj}
    \d t_j=\sum_k S_{jk}\,\d s_k, S_{jk}\in\Q(t,s).
\end{equation}
After substitution we obtain the linear Pfaffian system with the ``dependent variables'' $x_1,\dots,x_n$ and independent variables $s_1,\dots,s_l$, although the coefficients are not rational anymore: they explicitly involve the algebraic functions $t_j(s)$. However, this system can be transformed to a system with coefficients in $\Q(s)$ in a way similar to that used in the proof of Proposition~\ref{prop:alg-pf}.

Denote by $\Bbbk=\Q(s)$ the field of rational functions of the variables $s$. Reducing all equations to the common denominator, we can transform the Pfaffian system to the form
\begin{equation*}
    \d x_i=\frac1{\Delta(t)}\sum_{j,k} Q_{ijk}(t)\,x_j\d t_k,\qquad \Delta,Q_{ijk}\in\Bbbk[t_1,\dots,t_m]
\end{equation*}
with the coefficients $Q_{ijk}$ \emph{polynomial} in $t$. Without loss of generality, we may assume that all such polynomials are of degree $\le d-1$ by virtue of the identities \eqref{alg-tr}.

We first replace the denominator $\Delta(t)$ by an element from the field $\Bbbk=\Q(s)$, by iterated application of Lemma~\ref{lem:elim}. More specifically, we eliminate inductively the variables $t_1,\dots,t_k$ from the denominator, replacing them by the variables $s_j$, using the equations \eqref{alg-tr}. Then the standard division with remainder allows to exclude all powers of $t_j$ of degrees higher than $d_j=\deg_{t_j}P_j$. Finally, the lower degree powers $t_j,t_j^2,\dots,t_j^{d_j-1}$ are declared to be the ``new dependent variables'' governed by the obvious recurrent equations $\d (t_j^r)=r t_j^{r-1}\d t_j$ transformed to the linear Pfaffian form using the identities \eqref{dtj}.

The quasiunipotence the monodromy is proved as follows. Consider a small loop $\gamma:(\C,0)\to(\C P^l,a)$, defined by a holomorphic germ $s=s(z)$, and its composition with the algebraic map $\eta:z\mapsto t=t(s(z))$. Since the algebraic maps are in general multivalued, the latter map in general is multivalued (ramified at the origin) and its restriction on a small circle $|z|=\rho$ non-closed. However, the algebraic map \eqref{alg-tr} has only finitely many branches which are permuted along the small loop $\gamma$. Thus after finitely many iterations the small loop $\gamma^N$ would become a \emph{closed} loop in the $t$-space. On the level of the monodromy operators this means that the monodromy $M_\gamma$ of the composite system along the small loop $\gamma$ is quasiunipotent when raised to some finite power $N$, i.e., all eigenvalues of $M_\gamma^N$ are roots of unity. But then the same is true for the eigenvalues of $M_\gamma$ itself.
\end{proof}

\begin{Rem}
Inspection of the above proof shows that one can use a more general transformation of a ``triangular'' form, where the coordinate functions of the algebraic transformation are defined by the ``Pfaffian chain'' of the polynomial equations
\begin{equation}
  P(t_1,s)=0,\quad P_2(t_1,t_2,s)=0,\quad\dots\quad P_m(t_1,\dots,t_m,s)=0,\quad s=(s_1,\dots,s_l).
\end{equation}
It would be very interesting to know whether the restriction of a Q-system on an arbitrary projective subvariety $Z$ defined over $\Q$, can always be represented by a Q-system (with eventually larger number of the dependent variables). On the other hand, probably there are compositions of Q-functions, which are themselves not Q-functions.
\end{Rem}

\begin{Rem}
The possibility of introducing extra parameters and additional auxiliary equations (algebraic or Pfaffian) implies that in the definition of a Q-system one can replace the assumption that $\Omega$ is defined over $\Q$ by formally weaker but in fact equivalent condition that it is defined over the ring of algebraic numbers $\overline\Q$.
\end{Rem}

\section{Applications to the Hilbert 16th problem}

Although the above exposition was aimed at showing that the class of Q-functions is rich enough and important by itself to merit an investigation, still by far the strongest motivation for its study comes from the theory of planar real algebraic foliations and bifurcations of limit cycles.
We recall briefly the connection between these two areas. Detailed exposition can be found in numerous textbooks, among them \cite{thebook}.

\subsection{Crash course on perturbations of Hamiltonian foliations}\label{sec:IH16}
If $H\in\R[x,y]$ is a real polynomial in two variables with isolated critical points, the Pfaffian equation $\d H=0$ defines a real foliation with singularities on the plane $\R^2$ which extends to the projective plane $\R P^2$. Leaves of this foliation are connected components of the real level curves $\{H=t\}$. This foliation is integrable: the holonomy associated with each compact leaf (real algebraic oval) is trivial (identical), in particular, all nearby leaves are also compact and closed. In the classical language of the differential equations, there are no limit cycles for the corresponding system of Hamiltonian equations.

This changes if we apply a small one-parameter perturbation and consider the Pfaffian equation
\begin{equation}\label{pert-H}
    \d H+\e\omega=0,\qquad \omega=P\,\d x+Q\,\d y,\qquad P,Q,H\in\R[x,y],\quad \e\in(\R^1,0).
\end{equation}
Since the holonomy of a compact leaf is analytically dependent on the parameter $\e$, it can be expanded as a converging series. Using the function $H$ as the local chart on the transversal, we can write
\begin{equation}\label{poi}
    \eta(t,\e)=t+\e I_1(t)+\e^2 I_2(t)+\cdots,
\end{equation}
where $I_1(t),I_2(t),\dots$ are the first and the higher \emph{variations} of the holonomy, functions, defined on the nonsingular compact leaves of the initial integrable foliation and analytically depending on these leaves. A simple calculation known already to Poincar\'e shows that
\begin{equation}\label{ai}
    I_1(t)=\oint_{c_t} \omega, \qquad c_t\Subset\{H=t\}\subset\R^2
\end{equation}
is the real Abelian integral, cf.~with \eqref{ai-gen}. If the first variation is nontrivial, $I_1(t)\not\equiv0$, then the compact isolated leaves (limit cycles) of the perturbed foliation \eqref{pert-H} are rather faithfully tracked by the real isolated roots of the period \eqref{ai} (an accurate statement involves the positive distance from the singular locus of the integrable foliation). If $I_1\equiv0$, then one should find the lowest order nontrivial variation $I_k$, $k\ge 2$, and study its zeros.

This construction justifies the following ``linearized'' (``infinitesimal'') relaxed version of the Hilbert's 16th problem: find an upper bound for the number of isolated zeros of the Abelian integral \eqref{ai} in terms of the degrees $\deg H$ and $\deg \omega=\max(\deg P,\deg Q)$.

Clearly, the role of the two ingredients, the Hamiltonian and the perturbation form, is different: e.g., as a function of the form, the integral \eqref{ai} is linear, while its dependence on $H$ is very tricky. As an intermediate step, one can split the problem, fix $H$ and consider the problem for various 1-forms $\omega$ of different degrees.

One can show that for each fixed $H$ of degree $n$ one can find finitely many polynomial 1-forms which generate all integrals \eqref{ai} as a module over the ring $\C[t]$ of polynomials: for a generic $H$ one can take $(n-1)^2$ monomial forms $\omega_i$ whose differentials $\d \omega_i$ span the finite-dimensional quotient space $\bigland^2(\R^2)/\bigland^1(\R^2)\land\d H_n$, where $H_n$ is the principal homogeneous part of $H$. The assertion means that for any polynomial 1-form $\omega$ of degree $d$ and any continuous family of cycles $c_t\subseteq\{H=t\}$ there exist polynomials $q_j\in\R[t]$ such that
\begin{equation}\label{env}
    \oint_{c_t}\omega=\sum_j q_j(t)\cdot\oint_{c_t}\omega_j,\qquad q_j\in\R[t],\ \deg q_j\le d/n.
\end{equation}
This means that the Abelian integrals of \emph{all polynomial 1-forms} belong to the Picard--Vessiot differential extension field $\Bbbk (X)$, where $X$ is the period matrix of the $(n-1)^2$-tuple of 1-forms $\omega_j$ over all cycles on the complex algebraic curves $\{H=t\}\subset\C^2$.

\subsection{Linear bounds}
Theorem~\ref{thm:Q-pic-fuchs} allows to produce an upper bound for the number of isolated roots of a combination \eqref{env} which would be double exponential in $d$. However, for quite some time it was known that this number admits an \emph{asymptotic} estimate which is \emph{linear} in $d$: it does not exceed $C(n) d+C'(n)$, see \cite{petrov} and \cite[Theorem~6.26, pp.177--183]{zoladek}. Here the ``constant'' $C(n)$, which depends only on the degree $n=\deg H$, is completely explicit, but the other constant $C'(n)$ is only proved to exist, with the proof giving absolutely no clue on how it can be estimated.

It turns out that this is a general feature of Q-functions. Consider a Q-system \eqref{pfm} on $\C P^m$ and the module $\mathscr M(X)$ its solution matrix $X(t)$ generates over the field $\C(t_1,\dots,t_m)$ of rational functions on the domain of the Q-system: by definition, it consists of all linear combinations $\sum_{i,j}q_{ij}(t)X_{ij}(t)$ with \emph{rational} coefficients $q_{ij}(t)=q_{ij}(t_1,\dots,t_m)$. This representation, albeit not unique, induces a filtration of this module by the degrees of the coefficients $q_j$: we will say that an element $f\in \mathscr M(X)$ has degree $\delta=\deg_t f$, if it can be represented as a linear combination as above with $\deg g_{ij}\le \delta$ for all $i,j$.

\begin{Thm}[\cite{galgal}, to appear]\label{thm:gg}
For any fixed Q-system of degree $d$ dimension $n$ and complexity $s$ on $\C P^m$, the number of isolated zeros of any element of degree $\delta$ from the module $\mathscr M(X)$ is bounded by an explicit expression
\begin{equation}\label{gg-bd}
    C(n,m,d)\cdot\delta +C'(n,m,d,s).
\end{equation}
Here $C,C'$ are two explicit functions bounded as follows: the leading coefficient $C(n,m,d)$ grows no faster than a double exponential $2^{2^{\scriptstyle\mathrm{Poly}(m,d)}}$ independent of $n$, while the constant term $C'(n,m,d,s)$ is bounded by $s^{N(n,m,d)}$ with the degree $N$  bounded by a tower of \emph{five} exponentials in $(n,m,d)$.
\end{Thm}
The proof is achieved by combination of the construction of ``folding'' suggested by G.~Petrov and further elaborated by A.~Khovanskii, with the effective bound provided by Theorem~\ref{thm:Q-alg} above. It gives a fully constructive linear bound for the number of limit cycles, born by polynomial perturbations of any degree, from a given polynomial Hamiltonian foliation, as explained in \secref{sec:IH16}. Needless to say, this bound is also tremendously excessive by all expectations.

\subsection{Iterated integrals}
The higher variations $I_k(t)$ in \eqref{poi} can be computed recursively, see \cite{bautin,francoise}. For a generic Hamiltonian $H$ it can be shown that, under the inductive assumption that $I_1(t)\equiv\cdots\equiv I_{k-1}(t)\equiv 0$, the function $I_k(t)$ can be expressed as an Abelian integral of a polynomial 1-form of degree growing with $k$ \cite{gavrilov}, see also \cite{thebook}. Thus the number of limit cycles that can be tracked using variations of any given order $\le k$ (and, as before, distant from the singular locus), can be estimated in terms of $k$ by virtue of Theorem~\ref{thm:gg}. However, the problem of determining the maximal order $k$ such that the identical vanishing of all $I_1,\dots,I_k$ implies that the perturbation \eqref{pert-H} entirely consists of integrable systems, is a transcendentally difficult problem which includes as a particular case the famous Poincar\'e problem of discrimination between center and focus.

In the degenerate cases the higher variations $I_k$ may not be periods of polynomials forms (Abelian integrals). For instance, the monodromy operators may have Jordan cells of size greater than 2 (which is impossible for the ``genuine'' integrals over cycles of dimension 1), see \cite{gavr-ili}. Nevertheless, one can always express the first nonvanishing variation $I_k$ using iterated path integrals \cite{gavr-iter}. These are special functions which in some sense interpolate between (dual spaces to) finite-dimensional homology of the complex fibers $H^{-1}(t)$ and their infinite-dimensional homotopy. L.~Gavrilov in \cite{gavr-iter} proved that the iterated integrals considered as functions of a single variable $t$ satisfy a regular system of linear ordinary differential equations. The monodromy of this system was shown to be quasiunipotent in \cite{gavr-nov}. The explicit derivation of this system, achieved by S.~Benditkis and D.~Novikov, shows that this system is defined over $\Q$.

\begin{Thm}[S.~Benditkis  and D.~Novikov \cite{nov-bend}, submitted]
Iterated integrals of finite order are Q-functions.
\end{Thm}

 As in all preceding cases, the proof is constructive and gives an explicit bound for the parameters of the Q-function in terms of all relevant integer data, this time including the order of the iterated integrals. The bound is given by a tower function (iterated exponential) of height 4 and is obviously very excessive.

\subsection{Some open questions}
In a somewhat surprising way, all nontrivial examples of Q-functions originate from periods of rational forms and their generalizations (iterated integrals). It would be very interesting and instructive to have other types of examples, not directly related to periods or derived constructions.

Another question is motivated by Theorem~\ref{thm:gg}. A quasiunipotent representation $\rho$ of the fundamental group $\pi_1(\C P^m\ssm\S,\cdot)$ for an algebraic divisor $\S$ can be realized by a regular flat connection $\Omega$ (integrable Pfaffian system) with the preassigned monodromy $\rho$. The system $\Omega$ is unique modulo a rational gauge transformation (cf.~with p.~\pageref{p:gauge}) and under some conditions (to be determined, depending on the representation and the divisor) is defined over $\Q$, producing thus a family of Q-functions. Theorem~\ref{thm:gg} asserts that two globally gauge equivalent systems admit rather close (in the relative scale, of course) counting functions. The question is whether the counting function can be assigned to the pair $(\rho,\S)$ of the monodromy group and the polar divisor, rather than to a Q-system realizing this group.

Next, we note that the quasiunipotence assumption of Theorem~\ref{thm:main} can be relaxed: in fact, it is sufficient to require that all eigenvalues of the small monodromy operators have unit modulus. However, it may well be that such ``weakly quasiunipotent'' systems are necessarily quasiunipotent in the standard (strong) sense (e.g., by some application of the Gelfond--Schneider theorem).

Finally, the counting problem can be reformulated in geometric terms which admit a number of natural generalizations. A regular flat system on $\C P^m$ with $n$-dimensional fibers defines a (singular) foliation $\mathscr F$ with leaves of dimension $m$ and codimension $n$ which, because of the linearity, induces another foliation of dimension $m$ and codimension $n-1$ on the product $\C P^m\times\C P^{n-1}$. The counting function measures the number of isolated intersections between the leaves of this foliation and the (linear or projective, respectively) subspaces of the product, which have a complementary (co)dimension $n$ and very special position: projection of those subspaces on the base $\C P^m$ is one-dimensional, hence their intersection with the fibers is abnormally large.

One can try to drop this restriction and ask about the maximal number of isolated intersections with \emph{arbitrary} (linear or projective) subspace of the complementary (co)dimension. Analytically this means counting isolated solutions of \emph{systems} of equations of Q-functions. This problem seems to be quite challenging and may find applications in various counting problems. Besides, one can generalize the settings and consider arbitrary (not necessarily linear) foliations $\mathscr F$ defined by polynomial data and count isolated intersections of their leaves with affine subspaces of complimentary dimension.  So far the problem is solved only for foliations of dimension $1$ \cite{annalif-99} and of codimension $1$ \cite{fewnomials} or reducible to the latter case (``Pfaffian chains'').

%%%%% Enter the widest reference label as the first parameter%%%%%

%%% Only the title of article is italicized. No boldface numbers are used.%%%
%% The issue number is only given when the issues are paginated separately.%%%

\end{document}